\documentclass[12pt]{amsart}
\usepackage{amsmath,amssymb,amsfonts}

\newtheorem{definition}{Definition}[section]
\newtheorem{theorem}[definition]{Theorem}
\newtheorem{corol}[definition]{Corollary}
\newtheorem{lemma}[definition]{Lemma}

\newcommand{\F}{\mathbb{F}}
\newcommand{\fq}{\mathbb{F}_q}
\newcommand{\fpn}{\mathbb{F}_{p^n}}

\newcommand{\rmv}[1]{}

\begin{document}


\title{Cyclotomy and permutation polynomials of large indices }


\author{Qiang Wang}
\thanks{Research of the authors was partially supported
by NSERC of Canada}

\address{School of Mathematics and Statistics\\
Carleton University\\
Ottawa, ON K1S 5B6\\
Canada}
\email{wang@math.carleton.ca}




\maketitle

\date{}


\font\Bbb msbm10 at 12pt%




\begin{abstract}
We use cyclotomy to design new classes of permutation polynomials over finite fields. This allows us to generate many classes of permutation polynomials in an
algorithmic way. Many of them are permutation polynomials of large indices.
\end{abstract}


\section{Introduction}

Let $p$ be prime and $q=p^m$. Let $\fq$ be a finite field of $q$ elements and 
$\fq^* = \fq \setminus \{0\}$. A polynomial is a permutation
polynomial (PP) of a finite field $\mathbb{F}_q$ if it induces a
bijective map from ${\mathbb F}_q$ to itself. The study of
permutation polynomials of a finite field goes back to 19-th
century when Hermite and later Dickson pioneered this area of
research.  In recent years, interests in permutation polynomials
have significantly increased because of their
applications in coding theory and cryptography such as $S$-boxes.
 In some of these applications, the study of permutation polynomials  over finite fields has also been extended to the
study of permutation polynomials over finite rings and other
algebraic structures. For more background material on permutation
polynomials we refer to Chap. 7 of \cite{LN:97}. For a detailed
survey of open questions and recent results see \cite{LM:88},
\cite{LM:93}, \cite{Mullen:93}, and \cite{MullenWang:12}.

In \cite{AGW:11}, the authors provide a general theory which, in essence, reduces a problem of determining whether a given polynomial over a finite field $\fq$ is a permutation polynomial to a problem of determining whether another polynomial permutes a smaller set. One of very useful smaller sets is the set of cyclotomic cosets. Earlier,
Niederreiter and Winterhof \cite{NW:05} and  Wang\cite{Wang} have studied so-called cyclotomic permutations. Namely, let $C_0$ be a subgroup of $\mathbb{F}_q^*$ with
index $\ell \mid q-1$ and the factor group $\mathbb{F}_q^*/C_0$ consists of the {\it cyclotomic cosets}
$$ C_i := \gamma^i C_0, \  \  \  i = 0, 1, \cdots, \ell-1,$$
where $\gamma$ is a fixed primitive element of $\fq$. 
For any $A_0, A_1, \cdots, A_{\ell-1} \in \mathbb{F}_q$ and positive integer $r$,   the so-called {\it $r$-th order cyclotomic mapping $f^{r}_{A_0, A_1, \cdots,
A_{\ell-1}}$ of index $\ell$ }  from $\mathbb{F}_q$ to itself is defined by
$$
f^{r}_{A_0, A_1, \cdots,
A_{\ell-1}} (x) = 
\left\{
\begin{array}{ll} 
0, &   if ~ x=0; \\
A_0 x^{r}, &  if~ x \in C_0; \\
\vdots &  \vdots \\ 
A_i x^{r},   &   if ~x \in C_i; \\
\vdots & \vdots \\
A_{\ell-1} x^{r},  &  if~ x \in C_{\ell-1}.
\end{array}
\right.
$$

It is shown that  $r$-th order cyclotomic mappings produce the polynomials of the form  $x^r f(x^s)$ where $s =\frac{q-1}{\ell}$. Furthermore, PPs of the form $x^r f(x^s)$ have been intensively studied in \cite{AAW:08, AGW:09, AW:05, AW:06, AW:07, Chapuy:07, WL:91, Wang, Zieve-2, Zieve:09, Zieve}.

It is also well known that every polynomial $P(x)$ over $\mathbb{F}_q$ such that $P(0) =b$
has the form $ax^rf(x^s)+b$ with some positive integers
 $r, s$ such that $s\mid q-1$. Let $q-1 = \ell s$. 
 More precisely, we
observe that any polynomial $P(x)\in\mathbb{F}_q[x]$ can be
written as $a(x^rf(x^{(q-1)/\ell}))+b$, for some $r \geq 1$ and
$\ell\mid (q-1)$. To see this, without loss of generality, we can
write
$$P(x)=a(x^n+a_{n-i_1} x^{n-i_1}+\cdots+a_{n-i_k} x^{n-i_k})+b,$$
where $a,~a_{n-i_j}\neq 0$, $j=1, \cdots, k$.  Here we suppose
that $j\geq 1$ and $n-i_k=r$.
Then $P(x)=a\left(x^r f(x^{(q-1)/\ell}) \right)+b,$ where $f(x)=
x^{e_0}+a_{n-i_1} x^{e_1}+\cdots+ a_{n-i_{k-1}}x^{e_{k-1}} + a_{r}
$,
 $$\ell=\frac{q-1}{\gcd(n-r,n-r-i_1,\cdots, n-r-i_{k-1}, q-1)},$$
and $\gcd(e_0, e_1, \cdots, e_{k-1}, \ell)=1 .$ 
The constant $\ell$ is called the {\it index} of polynomial $P(x)$ (see \cite{AGW:09}).   The index of a polynomials is closely related to the concept of the least index of cyclotomic permutations. 
Many classes of PPs that are constructed recently have  small indices $\ell$, see for example, \cite{AAW:08, AW:05, AW:06, AW:07, Chapuy:07,  Zieve-2, Zieve:09}.

In this paper, we extend the definition of cyclotomic mappings and study the permutation polynomials corresponding to these cyclotomic mappings.  These polynomials have either  the presentation given in terms of  cyclotomic mappings of index $\ell$, 
$$
f^{r_0(x), r_1(x), \ldots, r_{\ell-1}(x)}_{A_0, A_1, \cdots,
A_{\ell-1}} (x) = 
\left\{
\begin{array}{ll} 
0, &   if ~ x=0; \\
A_0 r_0(x), &  if~ x \in C_0; \\
\vdots &  \vdots \\ 
A_i r_i(x),   &   if ~x \in C_i; \\
\vdots & \vdots \\
A_{\ell-1} r_{\ell -1}(x),  &  if~ x \in C_{\ell-1},
\end{array}
\right.
$$ 
or the polynomial presentation
\begin{equation*}\label{correspondence}
P(x) = \sum_{i=0}^{\ell-1} \frac{A_i}{\ell \zeta^{i(\ell-1)}} r_i(x) \left(x^{(\ell-1)s} + \zeta^i x^{(\ell-2)s} +\cdots + \zeta^{i(\ell-2)} x^s + \zeta^{i(\ell-1)} \right),
\end{equation*}
where $r_0(x), r_1(x), \ldots, r_{\ell-1}(x) \in \fq[x]$ and $\zeta =\gamma^s$ be a fixed  primitive $\ell$-th root of unity throughout this paper. 
Essentially, these polynomials are of the form $\sum_{i=0}^n x^{r_i} f_i(x^s)$. And indices of  polynomials of this form are normally large, which are different from $\ell$ in general.   After we study several cases when  $r_i(x)$'s are general and $\ell$ is small,  we study in detail the situation when $r_i(x)$'s are monomials $x^{r_i}$ for some positive integers $r_i$'s.  We give some general criteria of determining these polynomials are permutation polynomials of finite fields (Theorems~\ref{main},~\ref{construction}, ~\ref{main2}, ~\ref{specialMain}). One of them can be written as follows:

\begin{theorem}\label{main2}
Let $q-1 = \ell s$ and  $A_0, \ldots, A_{\ell-1} \in \fq$.  Then 
\[
P(x)  = 
\left\{
\begin{array}{ll} 
0, &   if ~ x=0; \\
A_0 x^{r_0}, &  if~ x \in C_0; \\
A_1 x^{r_1},  &  if~ x \in C_1; \\
\vdots & \vdots \\
A_{\ell-1} x^{r_{\ell-1}}, & if ~ x \in C_{\ell},
\end{array}
\right.
\]
is a PP of $\fq$ if and only if $(r_i, s) =1$ for any $i=0, 1, \ldots, \ell-1$ and  $\mu_{\ell} = \{ A_i^s \zeta^{r_i i}  \mid i=0, \ldots, \ell-1\}$, where $\mu_{\ell}$ is the set of all $\ell$-th roots of unity. 
\end{theorem}

We note that each $A_i^s \zeta^{r_i i}$ is an $\ell$-th root of unity as long as 
$A_i$ is not zero. Hence what we really need is to check all   $A_i^s \zeta^{r_i i}$ ($0\leq i \leq \ell-1$) are distinct in the above theorem. 
Using these criteria in different forms, we demonstrate our method by constructing many new classes of PPs (Theorems~\ref{ZhaHuThm8Gen}, ~\ref{PPoverfthreepower1}, ~\ref{PPoverfthreepower2}, ~\ref{PPoverfthreepower3}, ~\ref{PPoverfthreepower4}, \ref{specialR}, ~\ref{allOneBranchesCongruenceMinusOne}, ~\ref{binomialBranches},  Corollaries~\ref{rogersBranches}, ~\ref{exampleOfCongruenceOne},  ~\ref{exampleOfAllOneBranches}). 
Here we only list very few particular examples of these results over fields with small characteristic.

 \begin{theorem}
The polynomial $P(x) = x^{\frac{2(2^n-1)}{3}+2^i} + x^{\frac{2(2^n-1)}{3}+2^j} + x^{\frac{2^n-1}{3}+2^i} + x^{\frac{2^n-1}{3}+2^j} + x^{2^i}$ is a PP of $\F_{2^n}$ for any even positive integer $n$ and non-negative integers $i, j$. 
\end{theorem}

\begin{theorem}
The polynomial $f(x) =x^{\frac{3^n-1}{2} +3^i}+ 2 x^{\frac{3^n-1}{2} +3}+ 2 x^{\frac{3^n-1}{2} +2} + 2x^{\frac{3^n-1}{2} +1} + x^{3^i} + x^3 + x^2 + x$ is a PP of $\F_{3^n}$ for any positive integer $n$ and non-negative integer $i$. 
\end{theorem}

\begin{theorem}
The polynomial  $f(x) = x^{\frac{3^n-1}{2} +2^i}+ 2x^{\frac{3^n-1}{2} +3} + 2x^{\frac{3^n-1}{2} +2} + 2x^{\frac{3^n-1}{2} +1} + x^{2^i} + x^3 + x^2 +  x$ is a PP of $\F_{3^n}$ for any odd positive integer $n$ and non-negative integer $i$.
\end{theorem}

\begin{theorem}\label{PPoverfthreepower3}
The polynomial  $f(x) = x^{\frac{3^n-1}{2} +3^i}+ 2x^{\frac{3^n-1}{2} +3} + x^{\frac{3^n-1}{2} +2} + 2x^{\frac{3^n-1}{2} +1} + 2 x^{2^i} + 2x^3 + x^2 + 2 x$ is a PP of $\F_{3^n}$ for any odd positive integer $n$ and non-negative integer $i$.
\end{theorem}

\begin{theorem}
The polynomial  $f(x) = x^{\frac{3^n-1}{2} +2^i}+ x^{\frac{3^n-1}{2} +3} + 2x^{\frac{3^n-1}{2} +2} + x^{\frac{3^n-1}{2} +1} + 2x^{2^i} + x^3 + 2x^2 +  x$ is a PP of $\F_{3^n}$ for any odd positive integer $n$ and non-negative integer $i$.
\end{theorem}

\begin{theorem} \label{PPoverfthreepower4}
 Let $q=3^n$ and $\alpha, \beta, \gamma, \theta \in \mathbb{F}_{3^n}$. Let
$f(x) =  (\beta -\alpha) x^{(q-1)/2 +3} +  (\beta \theta -\alpha \gamma) x^{(q-1)/2 +2}  + (\beta \theta^2 -\alpha \gamma^2) x^{(q-1)/2 +1}  - (\beta + \alpha) x^{3} -  (\beta \theta + \alpha \gamma) x^{2}  - (\beta \theta^2 +\alpha \gamma^2) x$.  Then $f$ is  a PP of $\mathbb{F}_{3^n}$ if and only if  $\eta(\alpha) = \eta(\beta)$, $\eta(\gamma) =-1$, and $\eta(\theta) =1$. 
\end{theorem}

One can easily see from the definition that these classes of PPs have indeed large indices because they contain terms with consecutive exponents. We also remark that our results not only generalize  many previous results in \cite{AAW:08, AW:06, AW:07, Zieve-2}, but also generalize several more recent results including a class of PPs constructed by Hou in the study of reversed Dickson polynomials (Theorem 1.1 in \cite{Hou:11}) and several classes of PPs studied by Zha and Hu thereafter (Theorems 7-11 in \cite{ZhaHu}); see Theorems~\ref{PPoverfthreepower1}, ~\ref{3branchesZha1}, ~\ref{3branchesZha2}, \ref{squaredivisor}.  Our method can also provide an algorithmic way to generate permutation polynomials over finite fields.

\section{Cyclotomic mappings permutation polynomials}

Let $\gamma$ be a fixed primitive element of $\mathbb{F}_q$, $\ell \mid q-1$, and the set
of all nonzero $\ell$-th powers be $C_0 = \{ \gamma^{\ell j}: j = 0, 1,
\cdots, s -1\}$. Then $C_0$ is a subgroup of $\mathbb{F}_q^*$ of
index $\ell$. The elements of the factor group $\mathbb{F}_q^*/C_0$
are the {\it cyclotomic cosets}
$$ C_i := \gamma^i C_0, \  \  \  i = 0, 1, \cdots, \ell-1.$$
For any $A_0, A_1, \cdots, A_{\ell-1} \in \mathbb{F}_q$ and monic polynomials $r_0(x), \ldots,  r_{\ell-1}(x) \in \fq[x]$  we define a
{\it  cyclotomic mapping $f^{r_0(x), r_1(x), \ldots, r_{\ell-1}(x)}_{A_0, A_1, \cdots,
A_{\ell-1}}$ of index $\ell$ }  from $\mathbb{F}_q$ to itself by
$$
f^{r_0(x), r_1(x), \ldots, r_{\ell-1}(x)}_{A_0, A_1, \cdots,
A_{\ell-1}} (x) = 
\left\{
\begin{array}{ll} 
0, &   if ~ x=0; \\
A_0 r_0(x), &  if~ x \in C_0; \\
\vdots &  \vdots \\ 
A_i r_i(x),   &   if ~x \in C_i; \\
\vdots & \vdots \\
A_{\ell-1} r_{\ell -1}(x),  &  if~ x \in C_{\ell-1}.
\end{array}
\right.
$$
Moreover, $f^{r_0(x), r_1(x), \ldots, r_{\ell-1}(x)}_{A_0, A_1, \cdots,
A_{\ell-1}}$ is called an {\it cyclotomic mapping
 of the least index $\ell$} if the mapping can not be written as
  a cyclotomic mapping of any smaller index.
The polynomial of degree at most $q-1$ representing the cyclotomic
mapping $f^{r_0(x), r_1(x), \ldots, r_{\ell-1}(x)}_{A_0, A_1, \cdots,
A_{\ell-1}} (x) $ is called  an {\it  cyclotomic mapping polynomial}. In particular, when $r_0(x) = \cdots  = r_{\ell-1} (x) = x^r$ for a positive integer $r$, it
is  known  as a {\it $r$-th order cyclotomic mapping polynomial}, denoted by 
$f^r_{A_0, A_1, \cdots, A_{\ell-1}}(x)$ (see \cite{NW:05} for $r=1$ or \cite{Wang}).

Let $s = \frac{q-1}{\ell}$ and $\zeta =\gamma^s$ be a primitive $\ell$-th root of unity. It is shown in
\cite{AW:07}  that polynomials of the form $x^rf(x^s)$ and the $r$-th order
cyclotomic mapping polynomials $f^r_{A_0, A_1, \cdots, A_{\ell-1}}
(x)$ where  $A_i = f(\zeta^i)$ for $0\leq i\leq \ell-1$ are the same. More generally,
for any $P(x) = \sum_{i=0}^{\ell-1} r_i(x) f_i(x^s)$ with $P(0) =0$, we can also write $P(x)$ as a cyclotomic mapping as follows:

\[
P(x) = 
f^{P_0(x), P_1(x)\ldots, P_{\ell-1}(x)}_{A_0, A_1, \cdots, A_{\ell-1}} (x) 
= \left\{
\begin{array}{ll} 
0, &   \mbox{if} ~ x=0; \\
A_0 P_0(x), &  \mbox{if}~ x \in C_0; \\
\vdots &  \vdots \\ 
A_i P_i(x),   &   \mbox{if} ~x \in C_i; \\
\vdots & \vdots \\
A_{\ell-1} P_{\ell-1} (x),  &  \mbox{if}~ x \in C_{\ell-1},
\end{array}
\right.
\]
where $A_j P_j(x) = \sum_{i=0}^{\ell-1} r_i(x) f_i(\zeta^j)$ and $P_j(x)$ is the monic associated polynomial.  Indeed, for any $x \in C_j$, $x=\gamma^{\ell i+ j}$ for some $ 0\leq i \leq s-1$ and thus $P(\gamma^{\ell i+ j}) = \sum_{i=0}^{\ell-1} r_i(\gamma^{\ell i+ j}) f_i(\zeta^j) = A_j P_j(\gamma^{\ell i+ j})$.

On the other hand, for any given cyclotomic mapping of index $\ell$, 
$$
f^{r_0(x), r_1(x), \ldots, r_{\ell-1}(x)}_{A_0, A_1, \cdots,
A_{\ell-1}} (x) = 
\left\{
\begin{array}{ll} 
0, &   if ~ x=0; \\
A_0 r_0(x), &  if~ x \in C_0; \\
\vdots &  \vdots \\ 
A_i r_i(x),   &   if ~x \in C_i; \\
\vdots & \vdots \\
A_{\ell-1} r_{\ell -1}(x),  &  if~ x \in C_{\ell-1},
\end{array}
\right.
$$
we can find a unique polynomial $P(x)$ modulo $x^q -x$ corresponding to it. Namely,
\begin{equation}\label{correspondence}
P(x) = \sum_{i=0}^{\ell-1} \frac{A_i}{\ell \zeta^{i(\ell-1)}} r_j(x) \left(x^{(\ell-1)s} + \zeta^i x^{(\ell-2)s} +\cdots + \zeta^{i(\ell-2)} x^s + \zeta^{i(\ell-1)} \right).
\end{equation}
Indeed, if each $x\in C_i$, we must have $x^s = \zeta^i$. So we have 
$$
\frac{A_i}{\ell \zeta^{i(\ell-1)}} r_i(x) \left(x^{(\ell-1)s} + \zeta^i x^{(\ell-2)s} +\cdots + \zeta^{i(\ell-2)} x^s + \zeta^{i(\ell-1)} \right) = A_i r_i(x), 
$$
and for $j\neq i$, 
$$
\frac{A_j}{\ell \zeta^{j(\ell-1)}} r_j(x) \left(\zeta^{i(\ell-1)} + \zeta^j\zeta^{i(\ell-2)} +\cdots + \zeta^{j(\ell-2)} \zeta^i + \zeta^{j(\ell-1)} \right) = \frac{A_j (\zeta^{i\ell} - \zeta^{j\ell})}{\ell \zeta^{j(\ell-1)} (\zeta^i - \zeta^j)} r_j(x) = 0. 
$$

The correspondence ~(\ref{correspondence}) provides a way to construct permutation polynomials of finite fields. First of all, it is obvious to obtain the following result on cyclotomic mappings. 

\begin{lemma}\label{mainLemma}
 Let $f^{r_0(x), r_1(x), \ldots, r_{\ell-1}(x)}_{A_0, A_1, \cdots,
A_{\ell-1}} (x)$ be a cyclotomic mapping of index $\ell$ over $\fq$ given as 
$$
f^{r_0(x), r_1(x), \ldots, r_{\ell-1}(x)}_{A_0, A_1, \cdots,
A_{\ell-1}} (x) = 
\left\{
\begin{array}{ll} 
0, &   if ~ x=0; \\
A_0 r_0(x), &  if~ x \in C_0; \\
\vdots &  \vdots \\ 
A_i r_i(x),   &   if ~x \in C_i; \\
\vdots & \vdots \\
A_{\ell-1} r_{\ell -1}(x),  &  if~ x \in C_{\ell-1}.
\end{array}
\right.
$$
Then $f^{r_0(x), r_1(x), \ldots, r_{\ell-1}(x)}_{A_0, A_1, \cdots,
A_{\ell-1}}$ induces a permutation of $\fq$ if and only if  $\cup_{i=0}^{\ell-1} A_i r_i(C) = \fq^*$,  where $A_i r_i(C_i) = \{ A_i r_i(x) \mid x \in C_i \}$ for $ 0\leq i \leq \ell-1$.  In particular, if $A_0, \ldots, A_{\ell -1} \neq 0$ and each  $A_i r_i(x)$ is a bijective map from  $C_i$ to another coset $C_{j_i}$, then  $f^{r_0(x), r_1(x), \ldots, r_{\ell-1}(x)}_{A_0, A_1, \cdots,
A_{\ell-1}}$ induces a permutation of $\fq$ if and only if 
 $$\{ A_0 r_0(C_0), \ldots, A_{\ell-1} r_{\ell-1} (C_{\ell-1}) \} = \{C_0, \ldots, C_{\ell -1} \}.$$ 
\end{lemma}

Lemma~\ref{mainLemma} and Equation~(\ref{correspondence}) provide us a general scheme to construct PPs of finite fields. We can design PPs of specific type in two steps. First, we  choose an $\ell \mid q-1$ and set a pattern of permutation of cyclotomic cosets. For example, we may want to have a permutation which maps $C_0$ to $C_1$, $C_1$ to $C_2$, etc. Secondly, we choose different polynomials which maps one $C_i$ to another $C_j$ satisfying the previous requirements and use them to form a cyclotomic mapping. In the above example, we choose $\ell$ polynomials $A_i r_i(x)$ which map $C_0$ to $C_1$, $C_1$ to $C_2$, etc, respectively. The polynomial determined by Equation~ (\ref{correspondence}) is the desired one. We note that $A_i r_i(x)$'s do not need to be PPs of $\fq$, they only need to be bijective from one $C_i$ to another $C_j$ depending on the requirements. This gives a lot of flexibility and opens up a direction of studying polynomials that map one $C_i$ to another $C_j$ bijectively. In the rest of paper, we demonstrate our methodology and construct many  new PPs of finite fields, many of them have large indices.

First of all, we obtain the following result for cyclotomic mappings polynomials of index $2$ which follows directly from Lemma~\ref{mainLemma}. 

\begin{theorem}
Let  $A_0, A_1 \in \fq$ and $f_0(x), f_1 (x)$ be any two  polynomials of $\fq$ such that  
$$
f(x) = f^{f_0(x), f_1(x)}_{A_0, A_1} (x) = 
\left\{
\begin{array}{ll} 
0, &   if ~ x=0; \\
A_0 f_0(x), &  if~ x \in C_0; \\
A_1 f_1(x),  &  if~ x \in C_1.
\end{array}
\right.
$$
Then $f$ is  a PP of $\fq$ if  either one of the following holds.

(i) $A_0 f_0(C_0) = C_0$ and $A_1 f_1(C_1) = C_1$; or

(ii) $A_0 f_0(C_0) = C_1$ and $A_1 f_1(C_1) = C_0$. 

\end{theorem}

In particular, if we take $f_1(x), f_2 (x)$ as any two polynomials of $\fq$ of indices at most $2$, then we have the following.

\begin{theorem}
Let $q$ be odd and let $r_0, r_1$ be positive integers and $f_0(x), f_1(x) \in \fq[x]$. Let  
$$
f(x) =
\left\{
\begin{array}{ll} 
0, &   if ~ x=0; \\
x^r_0 f_0(x^{(q-1)/2}), &  if~ x \in C_0; \\
x^r_1 f_1(x^{(q-1)/2}),  &  if~ x \in C_1.
\end{array}
\right.
$$
Then $f$ is  a PP of $\fq$ if and only if  $(r_0, (q-1)/2) = (r_1, (q-1)/2) =1$
and $\eta (f_0(1) f_1(-1) ) = (-1)^{r_1 +1}$, where $\eta$ is a quadratic character
 of $\fq$. 
\end{theorem}
\begin{proof} Obviously $f_0(x^{(q-1)/2}) = f_0(1)$  for $x\in C_0$ and $f_1(x^{(q-1)/2}) = f_1(-1)$  for $x\in C_1$. If $f$ is a PP, we must have
$(r_0, (q-1)/2) = (r_1, (q-1)/2) =1$. Moreover, 
 $f_0(1) x^{r_0}$ maps $C_0$ onto $C_0$ if $\eta(f_0(1))=1$, and onto $C_1$ if $\eta(f_0(1))=-1$. Therefore $f_1(-1) x^{r_1}$ maps $C_1$ onto $C_1$ if $\eta (f_1(-1) ) = (-1)^{r_1 +1}$, and onto $C_0$ if $\eta (f_1(-1)) = (-1)^{r_1}$.  In any case, $\eta (f_0(1) f_1(-1) ) = (-1)^{r_1 +1}$. The converse is obvious and we omit the proof. 
\end{proof}

 The following result generalizes Theorem 8 \cite{ZhaHu} which only gives the sufficient part.

\begin{theorem} \label{ZhaHuThm8Gen}
Let $p$ be an odd prime and $n$, $t$, $r$ be any positive integers. Then 
$f(x) = (1-x^t)x^{\frac{p^n-1}{2} + r} -x^r - x^{t+r}$ is a PP over $\mathbb{F}_{p^n}$  if and only if $(r, p^n-1)=1$ and $(t+r,\frac{p^n-1}{2} ) =1$.
\end{theorem}

\begin{proof}
We rewrite $f(x) = (1-x^t)x^{\frac{p^n-1}{2} + r} -x^r - x^{t+r} = x^r (x^{\frac{p^n-1}{2}} -1 ) - x^{r+t} (x^{\frac{p^n-1}{2}} + 1 )$. Obviously $s= \frac{p^n-1}{2}$, $\ell = 2$, $r_0=t+r$, $f_0(x) = -(x+1)$, $r_1=r$ and $f_1(x) =x-1$. So we write $f(x) = f_{-2, -2}^{x^{r+t}, x^r} (x)$.  By the previous theorem, $f$ is  a PP if and only if $(r, (p^n-1)/2)=1$, $(t+r,(p^n-1)/2)=1$, and $(r, 2) =1$.  
\end{proof}

If $(t+r, \frac{p^n-1}{2}) =1$ and $(t+r, p^n-1)=2$, then we must have $p^n \equiv 3 \pmod{4}$. Hence we have the following corollary.
\begin{corol}[Theorem 8 \cite{ZhaHu}]
Let $p$ be an odd prime and $n$, $t$, $r$ be any positive integers. Then 
$f(x) = (1-x^t)x^{\frac{p^n-1}{2} + r} -x^r - x^{t+r}$ is a PP over $\mathbb{F}_{p^n}$ provided

(i) $(r, p^n-1)=1$ and $(t+r, p^n -1) =1$; or

(ii) $(r, p^n-1)=1$, $(t+r, p^n-1) =2$ and $p^n \equiv 3 \pmod{4}$.
\end{corol}

Next we  obtain the following new classes of PPs over finite fields of characteristic $3$.  

\begin{theorem} \label{PPoverfthreepower1}
Let $q=3^n$ and  $t$ be any positive integer. Let $\alpha, \beta, \theta \in \mathbb{F}_{q}^*$ and   
$$
f(x) =
\left\{
\begin{array}{ll} 
0, &   if ~ x=0; \\
\alpha x^t , &  if~ x \in C_0; \\
\beta ( x^3+\theta x^2+ \theta^2 x),  &  if~ x \in C_1.
\end{array}
\right.
$$
Then $f$ is  a PP of $\fq$ if and only if   $(t,  \frac{q-1}{2}) =1$, $\eta(\theta) =1$, and $\eta (\alpha) = \eta(\beta)$, where $\eta$ is the quadratic character of $\fq$. In this case, 
$f(x) = (\beta x^3 + \beta \theta x^2 + \beta \theta^2 x - \alpha x^t) x^{\frac{3^n-1}{2}} -
(\beta x^3 + \beta \theta x^2 + \beta \theta^2 x + \alpha x^t)$. 
\end{theorem}
\begin{proof}
Assume $f$ is a PP.  Because $x^t$ always map $C_0$ into $C_0$ and $\beta ( x^3+\theta x^2+ \theta^2 x) = \beta x (x-\theta)^2$, we must have $\eta (\alpha) = \eta(\beta)$. Indeed, we must have either  $\eta (\alpha) = \eta(\beta) = 1$  so that $f$ maps $C_0$ into $C_0$ and maps $C_1$ into $C_1$, or  $\eta (\alpha) = \eta(\beta) = -1$ so that $f$ maps $C_0$ into $C_1$ and maps $C_1$ into $C_0$. 
In either case, $(t, \frac{q-1}{2}) =1$ because  $x^t$ permutes $C_0$.  On the other hand, let $\beta ( x^3+\theta x^2+ \theta^2 x) = \beta ( y^3+\theta y^2+ \theta^2 y)$ for $x, y \in C_1$. Then we obtain $(x-y)(x^2 + (y-2\theta) x + (y-\theta)^2)=0$. It is obvious that $(x^2 + (y-2\theta) x + (y-\theta)^2)=0$ if and only if 
$\eta( (y-2\theta)^2 -4 (y-\theta)^2) = \eta(\theta y) =1$. Hence $\beta ( x^3+\theta x^2+ \theta^2 x)$ is one-to-one over $C_1$ if and only if $(x^2 + (y-2\theta) x + (y-\theta)^2) \neq 0$ over $C_1$. The latter is equivalent to
$\eta(\theta y) \neq 1$ and thus $\eta(\theta) = 1$.  The converse is similar and we omit the proof.
\end{proof}

We note that in the case that $\theta =0$, $f$ is a PP of $\fq$ if and only if $(t, \frac{q-1}{2})=1$ and   $\eta(\alpha) = \eta(\beta)$. 

In the study of permutation behaviour of the reversed Dickson polynomial, Hou \cite{Hou:11} proved that $D_{3^e+5}(1, x)$ is a PP over $\F_{3^e}$ when $e$ is positive even integer. Equivalently, Hou proved the following result which can also be put into the context of cyclotomic mappings. We observe that Hou's result follows from Theorem~\ref{PPoverfthreepower1} for $t=3$, $\alpha = \theta =2$ and $\beta=1$ (with a linear shift by $-1$). We note that  $\eta(2) = \eta(1) =1$ in $\mathbb{F}_{3^e}$ for any even positive $e$.

\begin{corol} [Theorem 1.1, \cite{Hou:11}] 
Let $e$ be a positive even integer. Then
$f(x) = (1-x-x^2)x^{\frac{3^e+1}{2}} -1 - x + x^2$ is a PP over $\F_{3^e}$. 
\end{corol}

Similarly, let $t=3^i$, $\alpha = \beta =2$ and $\theta=1$, we have the following result. 

\begin{theorem}
The polynomial $f(x) =x^{\frac{3^n-1}{2} +3^i}+ 2 x^{\frac{3^n-1}{2} +3}+ 2 x^{\frac{3^n-1}{2} +2} + 2x^{\frac{3^n-1}{2} +1} + x^{3^i} + x^3 + x^2 + x$ is a PP of $\F_{3^n}$ for any positive integer $n$ and non-negative integer $i$. 
\end{theorem}

In particular, when $i=1$, we have
\begin{corol}
The polynomial $f(x) = x^{\frac{3^n-1}{2} +2} + x^{\frac{3^n-1}{2} +1} +x^3 +2 x^2 +2 x$ is a PP over $\F_{3^n}$ for any positive integer $n$. 
\end{corol}

The following result (Proposition 1  in \cite{ZhaHu}) follows also from Theorem~\ref{PPoverfthreepower1} for $\alpha =1$. We note $(t, 3^n-1) =1$ implies that $t$ is odd and thus $(t, (3^n-1)/2) =1$. 

\begin{corol}
Let $t$ be a positive integer with $(t, 3^n-1) =1$. Assume $\theta, \beta \in \F_{3^n}^*$ with $\eta(\theta) = \eta(\beta) =1$. Then 
$f(x) = (\beta x^3 + \beta \theta x^2 + \beta \theta^2 x - x^t) x^{\frac{3^n-1}{2}} -
(\beta x^3 + \beta \theta x^2 + \beta \theta^2 x + x^t)$ is a PP over $\F_{3^n}$.
\end{corol}

Theorem~\ref{PPoverfthreepower1}  generalizes Proposition~1 in \cite{ZhaHu} in a few different ways.  First, it gives a necessary and sufficient description.  Secondly, a constant $\alpha$ could be interpreted as $f_i(x^{(q-1)/2})$ for any polynomial $f_i(x) \in \fq[x]$.  Thirdly, $t$ could be even if $n$ is odd because it is only required that $(t, (3^n-1)/2) =1$ in stead of $(t, 3^n-1)=1$. 
For example,  plug  $t=2^i$ and $\theta = \beta = \alpha =1$ in Theorem~\ref{PPoverfthreepower1} for $\F_{3^n}$ where $n$ is odd, we obtain

\begin{theorem}
The polynomial  $f(x) = x^{\frac{3^n-1}{2} +2^i}+ 2x^{\frac{3^n-1}{2} +3} + 2x^{\frac{3^n-1}{2} +2} + 2x^{\frac{3^n-1}{2} +1} + x^{2^i} + x^3 + x^2 +  x$ is a PP of $\F_{3^n}$ for any odd positive integer $n$ and non-negative integer $i$.
\end{theorem}

In particular, when $i=1$, we have 

\begin{corol}
The polynomial  $f(x) = x^{\frac{3^n-1}{2} +3} + x^{\frac{3^n-1}{2} +1}  + 2x^3 + x^2 +  2x$ is a PP of $\F_{3^n}$ for any odd positive integer $n$. 
\end{corol}

In a similar way,  we obtain the following result which extends the previous results.

\begin{theorem}\label{PPoverfthreepower2}
Let $q=3^n$ and $t$ be any positive integer.  Let $\alpha, \beta, \theta \in \mathbb{F}_{q}^*$ and   
$$
f(x) =
\left\{
\begin{array}{ll} 
0, &   if ~ x=0; \\
\beta ( x^3+\theta x^2+ \theta^2 x),  &  if~ x \in C_0;\\
\alpha x^t , &  if~ x \in C_1.
\end{array}
\right.
$$
Then $f$ is  a PP of $\fq$ if and only if   $(t, \frac{q-1}{2} ) =1$, $\eta(\theta) = -1$, and any one of the following holds: (i) $t$ is odd and $\eta (\alpha) = \eta(\beta)$;  (ii) $t$ is even and $\eta(\alpha) = -\eta(\beta)$. In this case, 
$f(x) = - (\beta x^3 + \beta \theta x^2 + \beta \theta^2 x - \alpha x^t) x^{\frac{3^n-1}{2}} - (\beta x^3 + \beta \theta x^2 + \beta \theta^2 x + \alpha x^t)$. 
\end{theorem}
\begin{proof}
Assume $f$ is a PP.  Obviously,  $(t, \frac{q-1}{2} ) =1$ because $x^t$ maps $C_0$ onto  either $C_0$ or $C_1$. Moreover, let $\beta ( x^3+\theta x^2+ \theta^2 x) = \beta ( y^3+\theta y^2+ \theta^2 y)$ for $x, y \in C_0$. Then we obtain $(x-y)(x^2 + (y-2\theta) x + (y-\theta)^2)=0$. It is obvious that $(x^2 + (y-2\theta) x + (y-\theta)^2)=0$ if and only if 
$\eta( (y-2\theta)^2 -4 (y-\theta)^2) = \eta(\theta y) =1$. Hence $\beta ( x^3+\theta x^2+ \theta^2 x)$ is one-to-one over $C_0$ if and only if $(x^2 + (y-2\theta) x + (y-\theta)^2) \neq 0$ over $C_0$. The latter is equivalent to
$\eta(\theta y) \neq 1$ and thus $\eta(\theta) = - 1$.   We now consider two cases of $t$.  If $t$ is odd, then $x^t$ maps $C_1$ onto $C_1$.  Because $\beta ( x^3+\theta x^2+ \theta^2 x) = \beta x (x-\theta)^2$, we must have $\eta (\alpha) = \eta(\beta)$.  If $t$ is even, then $x^t$ maps $C_1$ onto $C_0$. Because $( x^3+\theta x^2+ \theta^2 x) =  x (x-\theta)^2$ maps $C_0$ onto $C_0$, we must have $\eta(\alpha) = -\eta(\beta)$.  The converse is similar and we omit the proof.
\end{proof}

For $t=3^i$, $\theta =2$, and $\alpha = \beta =1$, we apply Theorem~\ref{PPoverfthreepower2} over $\mathbb{F}_{3^n}$ with odd $n$ to obtain the following result.

\begin{theorem}\label{PPoverfthreepower3}
The polynomial  $f(x) = x^{\frac{3^n-1}{2} +3^i}+ 2x^{\frac{3^n-1}{2} +3} + x^{\frac{3^n-1}{2} +2} + 2x^{\frac{3^n-1}{2} +1} + 2 x^{2^i} + 2x^3 + x^2 + 2 x$ is a PP of $\F_{3^n}$ for any odd positive integer $n$ and non-negative integer $i$.
\end{theorem}



For $t=2$, $\theta = \beta =2$, and $\alpha=1$,  we apply Theorem~\ref{PPoverfthreepower2} over $\mathbb{F}_{3^n}$ with odd $n$ to obtain the following result.

\begin{theorem}
The polynomial  $f(x) = x^{\frac{3^n-1}{2} +2^i}+ x^{\frac{3^n-1}{2} +3} + 2x^{\frac{3^n-1}{2} +2} + x^{\frac{3^n-1}{2} +1} + 2x^{2^i} + x^3 + 2x^2 +  x$ is a PP of $\F_{3^n}$ for any odd positive integer $n$ and non-negative integer $i$.
\end{theorem}

In particular, when $i=1$, we obtain

\begin{theorem}
The polynomial  $f(x) = x^{\frac{3^n-1}{2} +3} + x^{\frac{3^n-1}{2} +1} + x^3 + x^2 + x$ is a PP over $\F_{3^n}$ for any odd positive integer $n$. 
\end{theorem}

We also obtain the following PPs such that  both branches are cubic polynomials.

\begin{theorem} \label{PPoverfthreepower4}
 Let $q=3^n$ and $\alpha, \beta, \gamma, \theta \in \mathbb{F}_{3^n}$. Let
$f(x) =  (\beta -\alpha) x^{(q-1)/2 +3} +  (\beta \theta -\alpha \gamma) x^{(q-1)/2 +2}  + (\beta \theta^2 -\alpha \gamma^2) x^{(q-1)/2 +1}  - (\beta + \alpha) x^{3} -  (\beta \theta + \alpha \gamma) x^{2}  - (\beta \theta^2 +\alpha \gamma^2) x$.  Then $f$ is  a PP of $\mathbb{F}_{3^n}$ if and only if  $\eta(\alpha) = \eta(\beta)$, $\eta(\gamma) =-1$, and $\eta(\theta) =1$. 
\end{theorem}
\begin{proof}
Obviously, we have 
$$
f(x) =
\left\{
\begin{array}{ll} 
0, &   if ~ x=0; \\
\alpha ( x^3+\gamma x^2+ \gamma^2 x) , &  if~ x \in C_0; \\
\beta ( x^3+\theta x^2+ \theta^2 x),  &  if~ x \in C_1.
\end{array}
\right.
$$
Assume $f$ is  a PP of $\fq$. Because $\alpha ( x^3+\gamma x^2+ \gamma^2 x)= \alpha x (x-\gamma)^2$  and  $\beta ( x^3+\theta x^2+ \theta^2 x) = \beta x (x-\theta)^2$, they map $C_0$ and $C_1$ into different cosets respectively, as long as $\eta(\alpha) = \eta (\beta)$. Moreover, let $\beta ( x^3+\theta x^2+ \theta^2 x) = \beta ( y^3+\theta y^2+ \theta^2 y)$ with $x, y \in C_1$. Then we obtain $(x-y)(x^2 + (y-2\theta) x + (y-\theta)^2)=0$. It is obvious that $(x^2 + (y-2\theta) x + (y-\theta)^2)=0$ if and only if 
$\eta( (y-2\theta)^2 -4 (y-\theta)^2) = \eta(\theta y) =1$. Hence $\beta ( x^3+\theta x^2+ \theta^2 x)$ is one-to-one over $C_1$ if and only if $\eta(\theta) = 1$. Similarly, $\alpha ( x^3+\gamma x^2+ \gamma^2 x)$ is one to one over $C_0$ if and only if $\eta(\gamma) = - 1$.  
\end{proof}

For the rest of paper, we concentrate on refinement of  Lemma~\ref{mainLemma} with more branches. Obviously, $A_i \neq 0$ for all $i$'s if $f$  is a PP. Moreover,  if $r_i(x)$'s are of certain special formats then we can simplify Lemma~\ref{mainLemma} significantly.

One of the most natural choice is that $r_i(x) = x^{r_i}$ for $i=0, \ldots, \ell-1$. In this case,  we must have  $(r_i, s) =1$ in order for $P(x)$ to be a PP; otherwise, $|C_i^{r_i} | \neq s$, a contradiction. Hence we have the following result.

\begin{theorem} \label{main}
Let $\ell, s, r_0, \ldots, r_{\ell-1}$ be positive integers such that
$s = (q-1)/\ell$ and  $(r_i, s)=1$ for any $i=0, \ldots, \ell-1$. Let $q$ be prime power and $A_0, \ldots, A_{\ell-1} \in \fq^*$. 
Let 
\[
P(x) = 
f^{x^{r_0}, x^{r_1}, \ldots, x^{r_{\ell-1}}}_{A_0, A_1, \cdots, A_{\ell-1}} (x) 
= \left\{
\begin{array}{ll} 
0, &   if ~ x=0; \\
A_0 x^{r_0}, &  if~ x \in C_0; \\
\vdots &  \vdots \\ 
A_i x^{r_i},   &   if ~x \in C_i; \\
\vdots & \vdots \\
A_{\ell-1} x^{r_{\ell -1}},  &  if~ x \in C_{\ell-1}.
\end{array}
\right.
\]
Then the following are equivalent. 
\begin{enumerate}

\item[(a)] $P(x)$ is a PP of $\fq$;

 \item[(b)] $A_iC_{i r_i } \neq A_{i^\prime} C_{i^\prime r_{i^\prime}}$ for any $0 \leq i < i^\prime \leq \ell-1$, where the subscripts of $C_{i r_i}$ are taken modulo $\ell$. 

\item[(c)] $Ind_\gamma (\frac{A_i}{A_{i^\prime}}) \not\equiv r_{i^\prime} i^\prime
- r_i i \pmod{\ell}$ for any $0 \leq i < i^\prime \leq \ell-1$, where
$ind_\gamma (a)$ is residue class $b \bmod q-1$ such that
$a=\gamma^{b}$.

\item[(d)] $\{ A_0, A_1\gamma^{r_i}, \cdots, A_{\ell-1} \gamma^{(\ell-1)r_i} \}$ is a system
of distinct representatives of $\mathbb{F}_q^*/C_0$.

\item[(e)] $\{A_i^s\zeta^{i r_i} \mid i = 0, \cdots, \ell -1 \}$ is the set $\mu_{\ell}$ of all distinct $\ell$-th roots of
unity.

\item[(f)]  ${\displaystyle \sum_{i=0}^{\ell-1}
\zeta^{cr_i i} A_i^{cs} =0}$ for all $c=1, \cdots, \ell-1$.
 \end{enumerate}
\end{theorem}
\begin{proof}  The proof is similar to the proof of Theorem ~1 in \cite{Wang} and we include it for the sake of completeness. 

Since $ C_i  = \{\gamma^{\ell j +i}: j = 0, 1, \cdots, s -1\}$,  for
any two elements $x \neq y \in C_i$, we have $x = \gamma^{\ell j+i}$
and $y =\gamma^{\ell j^\prime +i}$ for some $0\leq j \neq j^\prime
\leq s-1$ . Since $(r_i, s) =1$, we obtain $A_i x^{r_i} = A_i
\gamma^{\ell r_i j+ i r_i } \neq A_i y^{r_i} = A_i \gamma^{\ell r_i j^\prime +ir_i}$.
Moreover, it is easy to prove that $C_0^{r_0} = C_0$ and more
generally $C_i^{r_i} = C_{i r_i}$ for any $0\leq i \leq \ell-1$. Hence $(a)$
and $(b)$ are equivalent.

Because $A_i\gamma^{i r_i }$ is a coset representative of $A_iC_{ir_i }$,
it is easy to see that $(c)$, $(d)$, and $(e)$ are
equivalent. Finally, since all of $A_0^s, A_1^s\zeta^{r_1}, \cdots,
A_{\ell-1}^s \zeta^{(\ell-1)r_{\ell-1}}$ are $\ell$-th roots of unity, $(e)$ means
that $A_0^s$, $A_1^s\zeta^{r_1}$, $\cdots$, $A_{\ell-1}^s
\zeta^{(\ell-1)r_{\ell-1}}$ are all distinct. By Lemma~2.1 in \cite{AW:07},  $(e)$ is equivalent to $(f)$. 
\end{proof}

This result generalizes Theorem 2.2 \cite{AW:07}, Theorem~1 \cite{Wang}, and Lemma~2.1 \cite{Zieve-2}, and all the consequences in these references.  Furthermore, we obtain the following result in terms of the polynomial presentation.

\begin{theorem}\label{construction}
Let $q$ be a prime power, $\ell \mid q-1$ and $s = (q-1)/\ell$. Let
 $\fq$ be a finite field of $q$ elements and $\zeta \in \fq$ be a primitive $\ell$-th root of unity. Let $A_0, \ldots, A_{\ell-1} \in \fq^*$. Then 
$$P(x) = \sum_{i=0}^{\ell-1}  \frac{A_i x^{r_i}}{\ell \zeta^{i(\ell-1)}} \left(x^{(\ell-1)s} + \zeta^i x^{(\ell-2)s} +\cdots + \zeta^{i(\ell-2)} x^s + \zeta^{i(\ell-1)} \right)$$
is  a PP of $\fq$ if and only if $(r_i, s) =1$ for all $i=0, \ldots, \ell-1$ and  $\{A_i^s\zeta^{i r_i} \mid i = 0, \cdots, \ell -1 \} = \mu_{\ell}$, where $\mu_{\ell}$ is the set of all $\ell$-th roots of unity. The latter condition is equivalent to that $\{ t_i + ir_i  \mid i=0, \ldots, \ell-1 \}$ is a complete set of residues modulo $\ell$, where $A_i^s = \zeta^{t_i}$ for $i=0, \ldots, \ell-1$. In particular, if  $A_0^s = A_1^s = \ldots =A_{\ell-1}^s \neq 0$, then
$P(x)$ is a PP of $\fq$ if and only if  $(r_i, s) =1$ for all $i=0, \ldots, \ell-1$ and  $\{ ir_i  \mid i=0, \ldots, \ell-1 \}$ is a complete set of residues modulo $\ell$. 
\end{theorem}
\begin{proof}
By Theorem~\ref{main} and Equation~(\ref{correspondence}), $P(x)$ is a PP of $\fq$ 
if and only if $(r_i, s) =1$ for all $i=0, \ldots, \ell-1$ and $\{A_i^s\zeta^{i r_i} \mid i = 0, \cdots, \ell -1 \}$ is the set $\mu_{\ell}$ of all distinct $\ell$-th roots of unity. Moreover, $\{A_i^s\zeta^{i r_i}  = \zeta^{t_i + i r_i} \mid i = 0, \cdots, \ell -1 \}= \mu_{\ell}$ is equivalent to that $\{ t_i + ir_i  \mid i=0, \ldots, \ell-1 \}$ is a complete set of residues modulo $\ell$. 
\end{proof}

Theorem~\ref{construction} provides a simple algorithmic way to construct PPs of $\fq$ with large indices. First, take any factor $\ell$ of $q-1$ and let $s= \frac{q-1}{\ell}$.  Then pick any $\ell$ positive integers $r_0, \ldots, r_{\ell-1}$ such that $(r_i, s) =1$ for $i=0, \ldots, \ell-1$ and any $\ell$ nonzero constants $A_0, \ldots, A_{\ell-1} \in \fq^*$.  As long as $A_i^s \zeta^{ir_i}$ ($0\leq i \leq \ell-1$) are all distinct (equivalently,  $\{ t_i + ir_i  \mid i=0, \ldots, \ell-1 \}$ is a complete set of residues modulo $\ell$), we obtain a PP of $\fq$.  In this way, one can construct a very large amount of classes of PPs of $\fq$. Here we give a few more examples of  PPs of finite fields produced by our construction method.

First we consider a few classes of PPs with three branches.
\begin{corol}\label{3branches}
Let $q=p^n$ such that $3\mid q-1$ and  $s= \frac{q-1}{3}$. Let $\zeta$ be a primitive $3$-rd root of unity.  Let $A_0, A_1, A_2 \in \fq$. 
Then $P(x) = A_0 x^{r_0} \left(x^{2s} +  x^{s} +1 \right) + 
\zeta A_1 x^{r_1}$ $\left(x^{2s} + \zeta x^{s} +\zeta^{2} \right)$ $+ \zeta^{2} A_2 x^{r_2} \left(x^{2s} + \zeta^2 x^{s} +\zeta \right) $ is a PP of $\fq$ if and only if  $(r_i, s) =1$ for $i=0, 1, 2$ and $ \{ A_0^s, A_1^s \zeta^{r_1}, A_2^s \zeta^{2r_2} \}$ $= \{ 1, \zeta, \zeta^2\}$.
\end{corol}

The  following result generalizes Theorem 9 in \cite{ZhaHu}. Again, we show these conditions are both necessary and sufficient.

\begin{theorem}\label{3branchesZha1}
Assume $p^n \equiv 1 \pmod{3}$. Let $s= \frac{p^n-1}{3}$ and $\zeta$ be an element of
$\mathbb{F}_{p^n}$ of order $3$. Then
\[
f(x) = x (x^s - \zeta)(x^s-\zeta^2) + x^3 (x^s - 1)(x^s-\zeta^2) + \zeta x^p 
(x^s - 1)(x^s-\zeta)
\]
is a PP over $\fpn$ if and only if

(a) $p \equiv 1 \pmod 3$ and $s  \equiv 1 \pmod 3$; or

(b) $p\equiv 2 \pmod 3$  and $s  \equiv 2 \pmod 3$. 



\end{theorem}

\begin{proof}
In this case, $\ell =3$ and $r_0 =1, r_1 =3, r_2=p$. Also 
$f_0(x) = (x- \zeta)(x-\zeta^2)$, $f_1(x) = (x-1)(x-\zeta^2)$, and $f_2(x) = \zeta (x-1)(x-\zeta)$. So $A_0 = (1-\zeta)(1-\zeta^2) = 1- \zeta - \zeta^2 + \zeta^3 = 2-\zeta -\zeta^2 = 3$, $A_1 = (\zeta-1)(\zeta-\zeta^2) = 3 \zeta^2$, and $A_2 =\zeta (\zeta^2 -1)(\zeta^2 -\zeta) = 3\zeta^2$.  Hence
\[
f(x) = f_{A_0, A_1, A_2}^{x^1, x^{3}, x^p}(x) = \left\{
\begin{array}{ll}
0  &  x =0;\\
3 x &  x\in C_0; \\
3\zeta^2 x^{3} & x\in C_1; \\
3\zeta^2 x^p  & x \in C_2.
\end{array} \right.
\]

Obviously, we have $(r_i, s) =1$ for $i=0, 1, 2$.  Moreover, 
$\{A_0^s, A_1^s \zeta^{3}, A_2^s \zeta^{2p} \} = \{3^s, 3^s \zeta^{2s+3}, 3^s \zeta^{2s+2p} \}$ is equal to  $\{1,  \zeta^{2s}, \zeta^{2s+2p} \}$ if and only if $p, s$ satisfy  either  $p \equiv 1 \pmod 3$ and $s  \equiv 1 \pmod 3$, or $p\equiv 2 \pmod 3$  and $s  \equiv 2 \pmod 3$.   By  Corollary~\ref{3branches}, we complete our proof. 
\end{proof}

The following result also generalizes Theorem 10 in \cite{ZhaHu}.

\begin{theorem}\label{3branchesZha2}
Let $i$ be any positive integer and assume $p^n \equiv 1 \pmod{9}$. 
Let $s= \frac{p^n-1}{3}$ and $\zeta$ be an element of
$\mathbb{F}_{p^n}$ of order $3$. Then
\[
f(x) = x (x^s - \zeta)(x^s-\zeta^2) + x^{p^i} (x^s - 1)(x^s-\zeta^2) + \zeta x^p 
(x^s - 1)(x^s-\zeta)
\]
is a PP over $\fpn$ if and only if

(i) $p \equiv 1 \pmod 3$; or

(ii) $i$ is odd and  $p\equiv 2 \pmod 3$.  
\end{theorem}

\begin{proof}

In this case, $\ell =3$ and $r_0 =1, r_1 =p^i, r_2=p$. Also 
$f_0(x) = (x- \zeta)(x-\zeta^2)$, $f_1(x) = (x-1)(x-\zeta^2)$, and $f_2(x) = \zeta (x-1)(x-\zeta)$. So $A_0 = (1-\zeta)(1-\zeta^2) = 1- \zeta - \zeta^2 + \zeta^3 = 2-\zeta -\zeta^2 = 3$, $A_1 = (\zeta-1)(\zeta-\zeta^2) = 3 \zeta^2$, and $A_2 = \zeta (\zeta^2 -1)(\zeta^2 -\zeta) = 3\zeta^2$.  Hence 
\[
f(x) = f_{A_0, A_1, A_2}^{x^1, x^{p^i}, x^p}(x) = \left\{
\begin{array}{ll}
0  &  x =0;\\
3 x &  x\in C_0; \\
3\zeta^2 x^{p^i} & x\in C_1; \\
3\zeta^2 x^p  & x \in C_2.
\end{array} \right.
\]

Obviously, we have $(r_j, s) =1$ for $j=0, 1, 2$.  Therefore, by Corollary~\ref{3branches}, $f(x)$ is  a PP over $\fq$ if and only if  $\{A_j^{s} \zeta^{r_j j} \mid j=0, 1, 2 \} = \{ 1,\zeta, \zeta^2\}$.  Indeed,  $\{A_j^{s} \zeta^{r_j j} \mid j=0, 1, 2 \} = \{ 3^s, (3\zeta^2)^s\zeta^{p^i}, (3\zeta^2)^s \zeta^{2p} \}$. We only need to find conditions so that  $1, \zeta^{2s+p^i}, \zeta^{2s+2p}$ are all distinct, equivalently,  $2s+p^i  \not \equiv  0 \pmod 3$, $s+p \not \equiv  0 \pmod{3}$ and $2s +p^i \not \equiv 2s+2p \pmod{3}$.  Under the assumption of $p^n \equiv 1 \pmod{9}$, we have $s \equiv 0 \pmod{3}$.  Hence we require $p^i  \not \equiv  0 \pmod 3$, $p \not \equiv  0 \pmod{3}$ and $p^i \not \equiv 2p \pmod{3}$. Therefore either $p \equiv 1 \pmod{3}$, or $p \equiv 2 \pmod{3}$ and $i$ is odd. 
\end{proof}

In particular, if $A_0, A_1, A_2$ belong to the same cyclotomic coset, then the condition $ \{ A_0^s, A_1^s \zeta^{r_1}, A_2^s \zeta^{2r_2} \}= \{ 1, \zeta, \zeta^2\}$ reduces $\{1,  \zeta^{r_1},  \zeta^{2r_2} \}= \{ 1, \zeta, \zeta^2\}$, which is equivalent to $r_1 \equiv r_2 \not\equiv 0 \pmod{3}$.

\begin{corol}
Let $q=p^n$ such that $3\mid q-1$ and  $s= \frac{q-1}{3}$. Let $\zeta$ be a primitive $3$-rd root of unity.  Let $A_0, A_1, A_2 \in \fq$ such that $A_0^s = A_1^s = A_2^s$. Then $P(x) = A_0 x^{r_0} $$ \left(x^{2s} +  x^{s}+ 1 \right) + 
A_1 x^{r_1}$ $\left( \zeta  x^{2s} + \zeta^2 x^{s} +1 \right)$ $+  A_2 x^{r_2} \left(\zeta^2 x^{2s} + \zeta x^{s} + 1\right) $ is a PP of $\fq$ if and only if $(r_i, s) =1$ for $i=0, 1, 2$ and $r_1 \equiv r_2 \not\equiv 0 \pmod{3}$. 
\end{corol}

From this corollary, if we take $q=2^n$ with $n$ is even,  $A_0=A_1= A_{2} = 1$,  $r_0 = 2^i$ and $r_1=r_{2} = 2^j$ for some non negative integers $i, j$, we obtain the following classes of PPs with coefficients in $\mathbb{F}_2$.

\begin{theorem}
The polynomial $P(x) = x^{\frac{2(2^n-1)}{3}+2^i} + x^{\frac{2(2^n-1)}{3}+2^j} + x^{\frac{2^n-1}{3}+2^i} + x^{\frac{2^n-1}{3}+2^j} + x^{2^i}$ is a PP over $\F_{2^n}$ for any even positive integer $n$ and non-negative integers $i, j$. 
\end{theorem}

Similarly, we can construct PPs with coefficients in general base field $\mathbb{F}_p$. 

\begin{theorem}
Let $q=p^m$, $\ell$ be a prime factor of $q-1$ with $s= \frac{q-1}{\ell}$.  Let $A_0, A_1\in \fq^*$. Then
$f(x) =A_0 x^{r_0} \left(x^{(\ell-1) s} + \cdots +  x^{s} +1 \right)$ $-
A_1 x^{r_1} \left(x^{(\ell-1)s} + \cdots +  x^{s} + \ell-1 \right)$ is a PP of $\fq$ if and only if $(r_0, s) = (r_1, s) =1$ and $A_0^s = A_1^s$.
\end{theorem}
\begin{proof}
Let $P(x)$ be the cyclotomic mapping $f_{A_0, A_1, \ldots, A_1}^{r_0, r_1, \ldots, r_1}(x)$, we obtain 
\begin{eqnarray*}
P(x) &=& \frac{A_0 x^{r_0}}{\ell} \left(x^{(\ell-1)s} +  x^{(\ell-2)s} +\cdots +  x^s + 1 \right) \\
& & +  \sum_{i=1}^{\ell-1}  \frac{A_i x^{r_i}}{\ell \zeta^{i(\ell-1)}} \left(x^{(\ell-1)s} + \zeta^i x^{(\ell-2)s} +\cdots + \zeta^{i(\ell-2)} x^s + \zeta^{i(\ell-1)} \right) \\
&=& \frac{A_0 x^{r_0}}{\ell} \left(x^{(\ell-1)s} +  x^{(\ell-2)s} +\cdots +  x^s + 1 \right) \\
& & + \frac{A_1 x^{r_1}}{\ell} \left( \sum_{i=1}^{\ell-1}  \zeta^{-i(\ell-1)} x^{(\ell-1)s} + \sum_{i=1}^{\ell-1}  \zeta^{-i(\ell-2)} x^{(\ell-2)s} +\cdots +  \sum_{i=1}^{\ell-1} \zeta^{-i} x^s + \ell -1 \right)\\
&=& \frac{A_0 x^{r_0}}{\ell} \left(x^{(\ell-1)s} +  x^{(\ell-2)s} +\cdots +  x^s + 1 \right) -  \frac{A_1 x^{r_1}}{\ell}  \left(x^{(\ell-1)s} + \cdots +  x^{s} + \ell -1 \right), 
\end{eqnarray*}
where the last equality holds because $\sum_{i=1}^{\ell-1}  \zeta^{-i(\ell-j)} = -1$ for all $j=1, \ldots, \ell-1$ when $\ell$ is prime. By Theorem~\ref{construction} and let $r_1=\cdots = r_{\ell-1}$ and $A_1 = \cdots = A_{\ell-1}$,  $P(x)$ is a PP of $\fq$ if and only if $(r_0, s) = (r_1, s) =1$ and $\{A_0^s, A_1^s \zeta^{r_1}, \ldots, A_1^s \zeta^{(\ell-1) r_1} \} = \mu_{\ell}$. The latter condition is equivalent to $A_0^s \neq A_1^s \zeta^{i r_1}$ for all $i=1, \ldots, \ell-1$, namely,  $A_1^s = A_0^s$.
\end{proof}
Taking $A_0 = A_1 =1$, we obtain the following PP with coefficients in the prime field $\mathbb{F}_p$. 

\begin{corol}
Let $q=p^m$, $\ell$ be a prime factor of $q-1$ with $s= \frac{q-1}{\ell}$.  Then
$f(x) = x^{r_0} \left(x^{(\ell-1) s} + \cdots +  x^{s} +1 \right)$ $-
 x^{r_1} \left(x^{(\ell-1)s} + \cdots +  x^{s} + \ell-1 \right)$ is a PP of $\fq$ if and only if $(r_0, s) = (r_1, s) =1$. 
\end{corol}

Finally we give another application of Theorem~\ref{construction}, which generalizes Theorem~11 in \cite{ZhaHu}.

\begin{theorem} \label{squaredivisor}
Assume  $p^n \equiv 1 \pmod{\ell^2}$ and let $\theta$ be an element of $\fpn$ of order $\ell$. Then 
\[
f(x) = \sum_{i=1}^t x^{p^i}  \prod_{j=1, j\neq i}^{t} (x^{(p^n-1)/\ell} - \theta^j)
\]
is a PP over $\fpn$ if and only if  $\{ip^i \pmod{\ell} \mid i=0, \ldots, \ell-1\} = \mathbb{Z}_{\ell}$. 
\end{theorem}

\begin{proof} Let $s = \frac{p^n-1}{\ell}$.   We note that $f$ is a cyclotomic mapping with
$r_i = p^i$ for $i=0, \ldots, \ell-1$ and  $ A_i = \prod_{j=1, j\neq i}^{t} (\theta^{i} - \theta^j) = \theta^{i(\ell-1)} (1-\theta^{-1})\cdots (1-\theta^{-(\ell-1)}) = \ell \theta^{i(\ell-1)} $.  By Theorem~\ref{construction}, $f$ is a PP of $\fq$ if and only if $(p^i, q-1)=1$ for all $i=0, \ldots, \ell-1$ and  $\{ \ell^s \theta^{i(\ell-1)s}  \theta^{ip^i} \mid i=0, \ldots, \ell-1\} = \mu_{\ell}$.  The condition $p^n \equiv 1 \pmod{\ell^2}$ implies $\ell \mid s$ and thus $\theta^{i(\ell-1)s} =1$. Hence $\{ \ell^s  \theta^{ip^i} \mid i=0, \ldots, \ell-1\} = \mu_{\ell}$ if and only if $\{ip^i \pmod{\ell} \mid i=0, \ldots, \ell-1\} = \mathbb{Z}_{\ell}$. 
\end{proof}

\begin{corol}[Theorem 11, \cite{ZhaHu}]\label{ZhaHuTheorem11}
Assume $p \equiv 1 \pmod{\ell}$ and $p^n \equiv 1 \pmod{\ell^2}$ and let $\theta$ be an element of $\fpn$ of order $\ell$. Then 
\[
f(x) = \sum_{i=1}^t x^{p^i}  \prod_{j=1, j\neq i}^{t} (x^{(p^n-1)/\ell} - \theta^j)
\]
is a PP over $\fpn$
\end{corol} 

\begin{corol}
Assume  $p^n \equiv 1 \pmod{16}$ and let $\theta$ be an element of $\fpn$ of order $\ell$. Then 
\[
f(x) = \sum_{i=1}^t x^{p^i}  \prod_{j=1, j\neq i}^{t} (x^{(p^n-1)/\ell} - \theta^j)
\]
is a PP over $\fpn$. 
\end{corol} 
\begin{proof}
Obviously, $p$ must be odd. If $p \equiv 1 \pmod{4}$, then $f$ is PP over $\fpn$ by Corollary~\ref{ZhaHuTheorem11}. If $p \equiv 3 \pmod{4}$, then $\{ip^i \mid i=0, 1, 2, 3\} = \{0, p, 2p^2, 3p^3\}$ is indeed a complete set of residue modulo $4$. By Theorem~\ref{squaredivisor}, $f$ is a PP over $\fpn$. 
\end{proof}

Similarly, we obtain the following corollary.
\begin{corol}
Assume  $p^n \equiv 1 \pmod{25}$ and let $\theta$ be an element of $\fpn$ of order $\ell$. Then 
\[
f(x) = \sum_{i=1}^t x^{p^i}  \prod_{j=1, j\neq i}^{t} (x^{(p^n-1)/\ell} - \theta^j)
\]
is a PP over $\fpn$ if and only if $p \equiv 1 \pmod{5}$.  
\end{corol}

\section{Realization of constants by polynomials}

In this section, we give more applications of Theorem~\ref{main} (or another version as in Theorem~\ref{main2}) to construct many new classes of PPs which have large indices and simple descriptions.  We  mainly consider how to choose constants $A_0, \ldots, A_{\ell-1}$ in terms of polynomials of specific formats in the cyclotomic mappings constructions. This demonstrate that our results generalize these results in \cite{AAW:08, AW:06, AW:07, Zieve-2}.

First of all, another way to rewrite  Theorem~\ref{main} is as follow:
\begin{theorem}\label{main2}
Let $q-1 = \ell s$,  $f_0(x), \ldots, f_{\ell-1} (x) \in \fq[x]$  and 
\[
P(x)  = 
\left\{
\begin{array}{ll} 
0, &   if ~ x=0; \\
x^{r_0} f_0(x^s), &  if~ x \in C_0; \\
x^{r_1} f_1(x^s),  &  if~ x \in C_1; \\
\vdots & \vdots \\
x^{r_{\ell-1}}  f_{\ell -1} (x^s), & if ~ x \in C_{\ell}.
\end{array}
\right.
\]
Then $P(x)$ is a PP of $\fq$ if and only if $(r_i, s) =1$ for any $i=0, 1, \ldots, \ell-1$ and  $\mu_{\ell} = \{\zeta^{r_i i}f_i(\zeta^i)^s  \mid i=0, \ldots, \ell-1\}$, where $\mu_{\ell}$ is the set of all $\ell$-th roots of unity. 
\end{theorem}

Using  Equation~(\ref{correspondence}) and Theorem~\ref{main2}, we can construct many PPs in the following polynomial format with large indices.
 
\begin{equation}\label{correspondence2}
P(x) = \sum_{i=0}^{\ell-1} \frac{x^{r_i} f_i(x^s)}{\ell \zeta^{i(\ell-1)}} \left(x^{(\ell-1)s} + \zeta^i x^{(\ell-2)s} +\cdots + \zeta^{i(\ell-2)} x^s + \zeta^{i(\ell-1)} \right).
\end{equation}

As long as $f_i(\zeta^i) \neq 0$, we can  rewrite Theorem~\ref{main2} as follow:

\begin{theorem}\label{specialMain}
Let  $q-1=\ell s$, $f_1(x), \ldots, f_{\ell-1}(x) \in \fq[x]$, and 
\[
P(x)  = 
\left\{
\begin{array}{ll} 
0, &   if ~ x=0; \\
x^{r_0} f_0(x^s), &  if~ x \in C_0; \\
x^{r_1} f_1(x^s),  &  if~ x \in C_1; \\
\vdots & \vdots \\
x^{r_{\ell-1}}  f_{\ell -1} (x^s), & if ~ x \in C_{\ell}.
\end{array}
\right.
\]
Suppose $f_i(\zeta^i)^s = A \zeta^{n_i}$ for each $i=0, \ldots, \ell-1$ and a nonzero constant $A \in \fq^*$. 
Then $P(x)$
is a PP of $\fq$ if and only if
\begin{enumerate}
\item[{\rm (i)}] $(r_i, s) =1$ for any $i=0, \ldots, \ell-1$.
\item[{\rm (ii)}] $\{ i r_i + n_i \mid i =0, \ldots, \ell-1\}$ is a complete set of residues modulo $\ell$. 
\end{enumerate}
In particular, if $f_i(\zeta^{i})^s = A$ for each $i=0,
\ldots, \ell-1$ and a nonzero constant $A\in \fq^*$, then 
$P$ is a permutation polynomial
of $\fq$ if and only if $(r_i, s) =1$ for  $i =0, \ldots, \ell -1$ and  $\{r_i i  \mid i =0, \ldots, \ell -1\}$  is a complete set of residues modulo $\ell$.
\end{theorem}

\begin{proof}{
Because the constant $A$ appears in each branch of the definition of $P(x)$, 
we can assume $A=1$ without loss of generality.   
From Theorem~\ref{main2}, we need to show that condition (ii) is
equivalent to   that 
$\mu_{\ell} = \{\zeta^{r_i i}f_i(\zeta^i)^s = \zeta^{ir_i + n_i}\mid i=0, \ldots, \ell-1\}$, which is obvious. 
}\end{proof}

We note that if $r_0 = r_1 = \cdots = r_{\ell-1} :=r$, then $\{r_i i  \mid i =0, \ldots, \ell -1\}$  is a complete set of residues modulo $\ell$ if and only if 
$(r, \ell) =1$. 
 If $r_0 = \cdots = r_{\ell-1}$ and $n_0 = \cdots = n_{\ell-1}$, we obtain Theorem 4.1 in \cite{AW:07} as a corollary.   As a
special case of Theorem~\ref{specialMain},  we also have the following
result. 

\begin{corol} \label{rogersBranches}
Let $q-1 = \ell s $,  $g_1(x), \ldots, g_{\ell-1}(x)$ be any $\ell$ polynomials over $\fq$, and 
\[
P(x)  = 
\left\{
\begin{array}{ll} 
0, &   if ~ x=0; \\
x^{r_0} g_0(x^s)^\ell, &  if~ x \in C_0; \\
x^{r_1} g_1(x^s)^\ell,  &  if~ x \in C_1; \\
\vdots & \vdots \\
x^{r_{\ell-1}}  g_{\ell -1} (x^s)^\ell, & if ~ x \in C_{\ell}.
\end{array}
\right.
\]
Then
$P(x)$ is a permutation polynomial of $\fq$ if and
only if $\{r_i i  \mid i =0, \ldots, \ell -1\}$  is a complete set of residues modulo $\ell$, $(r_i, s) =1$ and $g_i(\zeta^i) \neq 0$ for all $0\leq i \leq
\ell-1$.
\end{corol}
\begin{proof}
{This is true since if we set $f_i(x)= g_i(x)^\ell$, then we have
$f_i(\zeta^i)^s=g_i(\zeta^i)^{\ell s}=g_i(\zeta^i)^{q-1} =1$. The result
follows from Theorem \ref{specialMain}.}
\end{proof}

We note that earlier results of Wan and Lidl (see Corollary 1.4 in \cite{WL:91}), 
and Akbary and Wang (Theorem 3.1 in \cite{AW:07}) are also special cases of the above result.

We next construct cyclotomic permutations using classes of PPs with coefficients in some appropriate subfield which has been studied in \cite{AAW:08}, \cite{AW:06}, \cite{AW:07},  \cite{Chapuy:07},  and \cite{Zieve-2}. 

\begin{theorem} \label{specialR}
Let $\ell, r_0, \ldots, r_{\ell-1}$ be a positive integer with $q-1 = \ell s$. 
Suppose $q  = q_0^m$ where $q_0 \equiv 1 \pmod{\ell}$ and $\ell \mid m$.
Let $f_1(x), \ldots, f_{\ell-1}(x)$ be  polynomials in
$\mathbb{F}_{q_0}[x]$ 
and 
\[
P(x)  = 
\left\{
\begin{array}{ll} 
0, &   if ~ x=0; \\
x^{r_0} f_0(x^s), &  if~ x \in C_0; \\
x^{r_1} f_1(x^s),  &  if~ x \in C_1; \\
\vdots & \vdots \\
x^{r_{\ell-1}}  f_{\ell -1} (x^s), & if ~ x \in C_{\ell}.
\end{array}
\right.
\]
Then the polynomial $P(x)$ is
a permutation polynomial of $\F_q$ if and only if $\{r_i i  \mid i =0, \ldots, \ell -1\}$  is a complete set of residues modulo $\ell$, $(r_i, s)=1$ and $f_i(\zeta^i) \neq 0$ for all $0\leq i \leq \ell-1$.
\end{theorem}
\begin{proof} {
Let  $m=\ell n$.  The result is clear from Theorem \ref{specialMain}, since  we have
\begin{eqnarray*}
f_i(\zeta^i)^{\frac{q-1}{\ell}}&=&f_i(\zeta^i)^{\frac{q_0^{\ell n}-1}{\ell}}\\
&=& f_i(\zeta^i)^{\frac{q_0^{n}-1}{\ell} \left((q_0^{n})^{\ell-1}+(q_0^{n})^{\ell-2}+\cdots+1  \right)}\\
&=& \left( \prod_{j=0}^{\ell-1} f_i(\zeta^i)^{q_0^{nj}} \right)^{\frac{q_0^{n}-1}{\ell}}\\
&=& \left( f_i(\zeta^i)^\ell \right)^{\frac{q_0^{n}-1}{\ell}}\\
&=& 1.
\end{eqnarray*}
}\end{proof}

For example, let  $v$ be the order of $p$ in $\mathbb{Z}/\ell\mathbb{Z}$. For any positive integer $n$, we can take $q=q_0^m = p^{\ell vn}$ in the above result. We note Theorem~\ref{specialR} generalizes Corollary 3.3 (\cite{AW:07} or Theorem 3.1 (\cite{Chapuy:07}) or Theorem~1.2 (\cite{Zieve-2}), which deal with the case $P(x) = x^rf(x^s)$.

Moreover, in \cite{AAW:08, AW:07, Zieve-2}, classes of  PPs of the form $x^r(1+x^{e_1 s} + \ldots +x^{e_k s})^t$ are studied. For $h(x) = 1 + x+ \cdots +x^k$, it is well known that $h(\zeta^0) = k+1 \neq 0$ if and only if $p \nmid (k+1)$ and $h(\zeta^i) = \frac{\zeta^{(k+1)i}-1}{\zeta^i -1} \neq 0$ if and only if $\ell \nmid (k+1)i$ for $i=1, \ldots, \ell-1$. Here we construct cyclotomic permutations from these classes which generalizes Theorem~5.2 (\cite{AAW:08}) and Corollary~2.3 (\cite{Zieve-2}). 

\begin{corol}\label{exampleOfCongruenceOne}
Let $\ell$ be positive integer with $ q- 1 = \ell s$. 
For all $i=0, \ldots, \ell -1$, let  $r_i, k_i, e_i, t_i$  be positive integers such that  $(\ell, e_i) =1$ and $h_{k_i}(x) = 1 + x + \cdots + x^{k_i}$. 
Suppose $q  = q_0^m \pmod{\ell}$ such that $q_0 \equiv 1 \pmod{\ell}$ and $\ell \mid m$.  Then 
\[
P(x)  = 
\left\{
\begin{array}{ll} 
0, &   if ~ x=0; \\
x^{r_0} h_{k_0}(x^{e_0 s})^{t_0}, &  if~ x \in C_0; \\
x^{r_1} h_{k_1}(x^{e_1 s})^{t_1},  &  if~ x \in C_1; \\
\vdots & \vdots \\
x^{r_{\ell-1}}  h_{k_{\ell -1}} (x^{e_{\ell-1} s})^{t_{\ell-1}}, & if ~ x \in C_{\ell},
\end{array}
\right.
\]
permutes $\fq$ if and only if $\{r_i i  \mid i =0, \ldots, \ell -1\}$  is a complete set of residues modulo $\ell$, $(r_i, s)=1$ for all $0\leq i \leq \ell-1$, $p \nmid k_0+1$, and $\ell \nmid i(k_i +1)$ for all $i=1, \ldots, \ell-1$. 
\end{corol}

Now we construct several classes of PPs obtained from Theorem~\ref{specialMain} such that $f_i(\zeta^i)^s$ ($i=0, \ldots, \ell-1$) are not necessarily the same. Again, the following result extends Theorem 4.4 in \cite{AW:07} and Theorem ~1.3 in \cite{Zieve-2}.

\begin{theorem}\label{allOneBranchesCongruenceMinusOne}
Let $\ell$ be positive integer with $ q- 1 = \ell s$. 
For all $i=0, \ldots, \ell -1$, let  $r_i, k_i, e_i, t_i, n_i$  be positive integers such that  $(\ell, e_i) =1$. Put $h_{k^\prime_i}(x) = 1 + x + \cdots + x^{k^\prime_i}$ and $h_{k_i}(x) = 1 + x + \cdots + x^{k_i}$. Let $ \bar{k_i} = \ell/(\ell, k_i)$. 
Suppose $q  = q_0^m$ such that $q_0 \equiv -1 \pmod{\ell}$ and $m$ is even. 
Pick $\hat{h}_i \in \mathbb{F}_{q_0} [x]$  and let $f_i (x) := h_{k^\prime_i}(x)^{t_i} \hat{h}_i (h_{k_i}(x)^{\bar{k_i}})$. Then 
\[
P(x)  = 
\left\{
\begin{array}{ll} 
0, &   if ~ x=0; \\
x^{r_0} f_0(x^{e_0 s}), &  if~ x \in C_0; \\
x^{r_1} f_1(x^{e_1 s}),  &  if~ x \in C_1; \\
\vdots & \vdots \\
x^{r_{\ell-1}}  f_{\ell -1} (x^{e_{\ell-1} s}), & if ~ x \in C_{\ell},
\end{array}
\right.
\]
permutes $\fq$ if and only if $\{ (r_i+\frac{e_ik_i^\prime t_is}{2}) i  \mid i =0, \ldots, \ell -1\}$  is a complete set of residues modulo $\ell$, $(r_i, s)=1$ and 
$f_i(\zeta^i) \neq 0$ for all $0\leq i \leq \ell-1$. 
\end{theorem}
\begin{proof}
Let $m=2n$. We note that $\ell \mid q_0 +1$ implies that $\zeta^{q_0} = \zeta^{-1}$
and thus $h_{k_i^\prime}(\zeta^{ie_i})^{q_0-1} =$ 
$\left(   \frac{ \zeta^{(k_i^\prime +1) i e_i q_0} - 1}{\zeta^{i e_i q_0} -1} \right) 
\left (\frac{ \zeta^{i e_i} -1}{ \zeta^{(k_i^\prime +1) i e_i} - 1} \right)$ 
$= \zeta^{-k_i^\prime i e_i}$. Furthermore, $q_0 -1 \mid \frac{q_0^2-1}{q_0+1} \mid \frac{q-1}{q_0+1} \mid \frac{q-1}{\ell} =s$ implies that $h_{k_i^\prime}(\zeta^{ie_i})^{s} = \zeta^{-\frac{k_i^\prime i e_i s}{q_0-1}} = \zeta^{\frac{k_i^\prime i e_i s}{2}}$. 
Similarly, $h_{k_i}(\zeta^{ie_i})^{\bar{k_i} q_0}  = \left( \frac{h_{k_i}(\zeta^{ie_i})}{\zeta^{k_i i e_i}} \right)^{\bar{k_i}} = h_{k_i}(\zeta^{ie_i})^{\bar{k_i}}$ implies that $h_{k_i}(\zeta^{ie_i})^{\bar{k_i}} \in \mathbb{F}_{q_0}$.    
Then the result follows from Theorem \ref{specialMain}, since  we have
\begin{eqnarray*}
f_i(\zeta^{ie_i})^{\frac{q-1}{\ell}}&=& \left(h_{k^\prime_i}(\zeta^{ie_i})^{t_i} \right)^{\frac{q-1}{\ell}} \left( \hat{h}_i (h_{k_i}(\zeta^{ie_i})^{\bar{k_i}}) \right)^{\frac{q_0^{2n}-1}{\ell}}\\
&=& h_{k^\prime_i}(\zeta^{ie_i})^{t_i s} \left( \hat{h}_i (h_{k_i}(\zeta^{ie_i})^{\bar{k_i}} ) \right)^{\frac{q_0^{2}-1}{\ell} \left((q_0^{2})^{n-1}+(q_0^{2})^{n-2}+\cdots+1  \right)}\\
&=&  \zeta^{\frac{i e_i  k_i^\prime t_i s}{2}}\left( \prod_{j=0}^{n-1} \hat{h}_i (h_{k_i}(\zeta^{ie_i})^{\bar{k_i}} )^{q_0^{2j}} \right)^{\frac{q_0^{2}-1}{\ell}}\\
&=& \zeta^{\frac{i e_i  k_i^\prime t_i s}{2}},
\end{eqnarray*}
as long as $f_i(\zeta^{ie_i}) \neq 0$.
\end{proof}

Again, for $h(x) = 1 + x+ \cdots +x^k$, it is well known that $h(\zeta^0) = k+1 \neq 0$ if and only if $p \nmid (k+1)$ and  that  $h(\zeta^i) = \frac{\zeta^{(k+1)i}-1}{\zeta^i -1} \neq 0$ if and only if $\ell \nmid (k+1)i$ for $i=1, \ldots, \ell-1$.  We therefore obtain a generalization of Theorem 4.4 (\cite{AW:07}) and Corollary~2.4 (\cite{Zieve-2}) as follows:

\begin{corol}\label{exampleOfAllOneBranches}
Let $\ell$ be positive integer with $ q- 1 = \ell s$. 
For all $i=0, \ldots, \ell -1$, let  $r_i, k_i, e_i, t_i, n_i$  be positive integers such that  $(\ell, e_i) =1$. Put $h_{k_i}(x) = 1 + x + \cdots + x^{k_i}$.
Suppose $q  = q_0^m$ such that $q_0 \equiv -1 \pmod{\ell}$ and $m$ is even. 
Then 
\[
P(x)  = 
\left\{
\begin{array}{ll} 
0, &   if ~ x=0; \\
x^{r_0} h_{k_0}(x^{e_0 s})^{t_0}, &  if~ x \in C_0; \\
x^{r_1} h_{k_1}(x^{e_1 s})^{t_1},  &  if~ x \in C_1; \\
\vdots & \vdots \\
x^{r_{\ell-1}}  h_{k_{\ell -1}} (x^{e_{\ell-1} s})^{t_{\ell-1}}, & if ~ x \in C_{\ell},
\end{array}
\right.
\]
permutes $\fq$ if and only if $\{ (r_i+\frac{e_ik_it_is}{2}) i  \mid i =0, \ldots, \ell -1\}$  is a complete set of residue modulo $\ell$, $(r_i, s)=1$  for all $0\leq i \leq \ell-1$, $p \nmid k_0+1$, and $\ell \nmid i(k_i +1)$ for all $i=1, \ldots, \ell-1$. 
\end{corol}

Finally we take all branches as binomials and obtain a large class of PPs, which generalizes Theorem 3.1 \cite{AW:06} and Theorem~2.5 \cite{Zieve-2}. We note the necessary and sufficient description of a subclass of permutation binomials can be found in \cite{Wang, Wang2}.

\begin{theorem}\label{binomialBranches}
Let $\ell$ be positive integer with $ q- 1 = \ell s$. 
Let $u_i > r_i >0$ and $a_i \in \fq^*$ such that  $\gcd(u_i-r_i, q-1) := s$ is a constant for all $i=0, \ldots, \ell-1$. Let  $e_i := (u_i-r_i)/\ell$ and  $\eta$ be a fixed primitive $2\ell$-th root of unity in the algebraic closure of
$\fq$ and $\zeta=\eta^2$.   Suppose $(\eta^{i e_i} + a_i/\eta^{i e_i} )^s =1$ for each $i=0, \ldots, \ell-1$.  Then 
\[
P(x)  = 
\left\{
\begin{array}{ll} 
0, &   if ~ x=0; \\
x^{u_0} + a_0 x^{r_0}, &  if~ x \in C_0; \\
x^{u_1} + a_1 x^{r_1},  &  if~ x \in C_1; \\
\vdots & \vdots \\
x^{u_{\ell-1}} + a_{\ell-1} x^{r_{\ell-1}}, & if ~ x \in C_{\ell}.
\end{array}
\right.
\]
permutes $\fq$ if and only if $-a_i \neq \zeta^{i e_i}$ and $(r_i, s)=1$ for all $0\leq i \leq \ell-1$, $\{r_i i +  \frac{e_i s i}{2}  \mid i =0, \ldots, \ell -1\}$  is a complete set of residues modulo $\ell$.
\end{theorem}

\begin{proof}
Let $x^{u_i} + a_i x^{r_i} = x^{r_i} (x^{e_i s} + a)$. We have
\begin{eqnarray*}
(\zeta^{i e_i} + a)^{s}&=& (\eta^{2ie_i} + a)^s\\
&=& \eta^{ie_i s} \left( \eta^{ie_i} + a /\eta^{ie_i} \right)^s \\
&=& \eta^{ie_is} = \zeta^{ie_is/2}.
\end{eqnarray*}
The rest of proof follows easily from Theorem~\ref{specialMain}.
\end{proof}

\section{conclusion}

In this paper we study permutation polynomials of finite fields in terms of cyclotomy.  We provide both theoretical and algorithmic ways to generate permutation polynomials of finite fields. We have demonstrated how to construct concrete classes of PPs using our method. One can expect to generate more concrete classes of permutation polynomials by taking different polynomials as branches in our cyclotomic mapping construction. It is also  expected to further extend our method to additive cyclotomy as studied in \cite{AGW:11, YuanDing:12, Zieve}.

\medskip\par

\noindent School of Mathematics and Statistics, Carleton University, 1125 Colonel By Drive, Ottawa, Ontario, K1S 5B6, CANADA\\
E--mail address: {wang@math.carleton.ca}

\end{document}